\documentclass[reqno]{amsart}%
\usepackage{amsfonts}
\usepackage{amsmath}
\usepackage{amssymb}
\usepackage{graphicx}
\usepackage{amsthm}
\usepackage{xcolor}
\usepackage{hyperref}%
\providecommand{\U}[1]{\protect\rule{.1in}{.1in}}
\hypersetup{
colorlinks=true,     linktoc=all,         linkcolor=blue,  }
\newtheorem{theorem}{Theorem}
\theoremstyle{plain}

\newtheorem{lemma}{Lemma}

\newtheorem{proposition}{Proposition}
\newtheorem{remark}{Remark}

\numberwithin{equation}{section}
\begin{document}
\title{Counting fixed points free vector fields on $\mathbb{B}^{2}$}
\author{Simeon Stefanov}
\address{University of Architecture, Civil Engineering and Geodesy (UACEG), Sofia 1046,
1~Hr. Smirnenski Blvd.}
\email{sim.stef.ri@gmail.com}
\subjclass[2010]{Primary 57M50; Secondary 05A15}
\keywords{vector field, diagram, Catalan number, OEIS}

\begin{abstract}
The purpose of this elementary note is to count stationary points free general
position vector fields on the unit disk $\mathbb{B}^{2}$, depending on the
number of touch points between the flow lines and the boundary. Of course, we
are counting the number of classes under a natural equivalence relation on
such fields. It turns out that the number of such objects (diagrams) may be
calculated via some formula involving Catalan numbers, moreover, the
corresponding sequence may be identified as sequence A275607 from OEIS (the
On-line Encyclopedia of Integer Sequences). Catalan numbers appear while
counting various objects of geometric nature. Finally, an algorithm for
finding all such diagrams under consideration, is proposed. It is based on
enumerating all independent vertex sets of some highly symmetric graph of
geometric origin.

\end{abstract}
\maketitle

\section{Preliminary}

A generic (smooth) vector field without stationary points in the disk
$\mathbb{B}^{2}$ may be represented by a simple diagram as depicted at
\hyperref[f1]{Fig.~\ref*{f1}}. Here we are identifying the vector field with
the induced flow on the disk, so we shall classify fields by investigating the
corresponding flow diagrams.

\begin{center}
\begin{figure}[ptb]
\includegraphics[width=90mm]{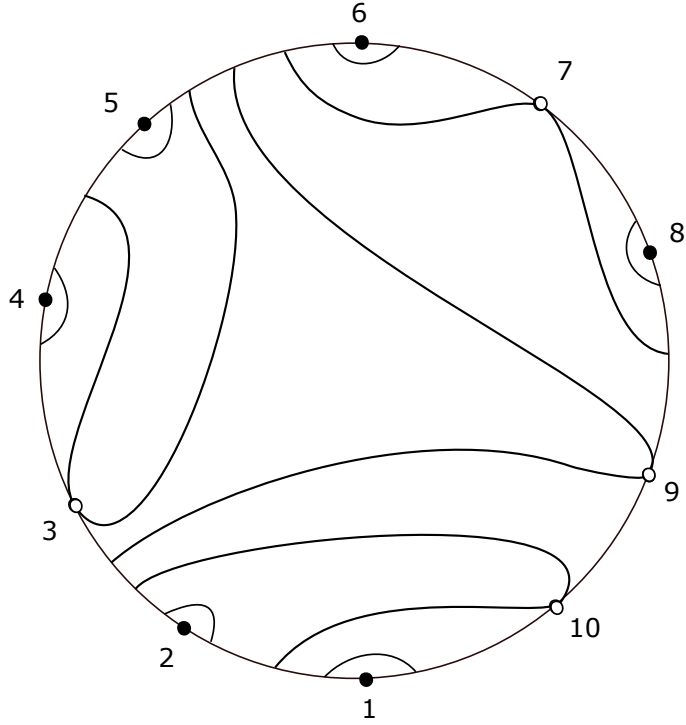}\caption{A generic fixed points free
vector field in $\mathbb{B}^{2}$.}%
\label{f1}%
\end{figure}
\end{center}

Let us start with some simple observations about general position fields in
$\mathbb{B}^{2}$:

a) A flow line is touching the boundary $\mathbb{S}^{1}$ in at most 1 point.

b) The number of touch points between all the flow lines and the boundary is finite.

c) Suppose there are $m$ touch points between the flow lines and
$\mathbb{S}^{1}$, yet there exist $m+2$ other exceptional points on the
boundary where the flow lines diagram is formed locally by \textquotedblleft
semi-circles\textquotedblright. We shall refer to latter points as
$a$\textit{-nodes}, while the touch points between level lines and the
boundary will be referred to as $b$\textit{-nodes}. For example, at
\hyperref[f1]{Fig.~\ref*{f1}} the total number of nodes is 10 and \ we have 6
$a$-nodes (1,2,4,5,6,8) and 4 $b$-nodes (3,7,9,10). (In some sense, $a$-nodes
may be thought of as \textquotedblleft sources\textquotedblright, while
$b$-nodes - as \textquotedblleft sinks\textquotedblright.) Now, the main
observation is that%
\[
\#(a\text{-nodes})-\#(b\text{-nodes})=2\text{.}%
\]

This equality becomes evident after drawing several examples by hand, but also
may be proved rigorously either combinatorially, or by counting the degree of
the vector field along the boundary $\mathbb{S}^{1}$ (it should equal 0, since
the field is stationary points free). Note that the total number of nodes
$2m+2$ is even.

In view of the above remarks, we may now settle the frame for counting fields.
We have some fixed set $A$ of $n$ points on the circle $\mathbb{S}^{1}$, where
$n$ is even%
\[
A=\left\{  p_{1},\dots,p_{n}\right\}  .
\]

We shall call points $p_{i}$ \textit{nodes}. (Without loss of generality we
may suppose $p_{i}$ to be the vertices of a fixed regular $n$-gon.) Then we
are considering the class $\mathcal{V}_{n}$ of vector fields such that any
field $v\in\mathcal{V}_{n}$ has $\frac{n}{2}+1$ $a$-nodes and $\frac{n}{2}-1$
$b$-nodes, all of them lying in set $A$. Now we define some natural
equivalence relation in $\mathcal{V}_{n}$:

$v_{1}\sim v_{2}$ if there is an orientation preserving diffeomorphism
$\varphi:\mathbb{B}^{2}\rightarrow\mathbb{B}^{2}$ such that $\varphi|_{A}=id$
(the nodes are fixed) and $\varphi$ is transforming the flow lines of $v_{1}$
onto those of $v_{2}$.

Let $D_{n}=\widetilde{\mathcal{V}}_{n}$ be the corresponding quotient set.\ It
is clear that the equivalence classes are finite in number and may be
visualized by simple diagrams. Set%
\[
T_{n}=\left\vert D_{n}\right\vert ,
\]

so, $T_{n}$ is the number of diagrams with $n$ nodes.

The main purpose of this article is to find the numbers $T_{n}$. We shall
prove that
\begin{equation}
T_{n}=3^{\frac{n}{2}-2}\left(  C_{\frac{n}{2}}+2C_{\frac{n}{2}-1}\right)  ,
\label{0}%
\end{equation}

where $C_{m}$ is the $m$-th Catalan number (cf. \cite{b1}, \cite{b2},
\cite{b3})
\[
C_{m}=\frac{(2m)!}{(m+1)!m!}\text{ .}%
\]
Setting $n=2k$, we get a more compact formula%
\begin{equation}
T_{2k}=3^{k-2}\left(  C_{k}+2C_{k-1}\right)  . \label{1}%
\end{equation}

Taking into account that the first Catalan numbers for $m=0,1,2,\dots$ are%

\[
\mathbf{1,1,2,5,14,42,132,429,\dots,}%
\]

it is not difficult to calculate the first terms of sequence $T_{2k}$ for
$k\geq1$:%
\[
\mathbf{1,4,27,216,1890,17496,168399,\dots}\text{,}%
\]

\begin{center}
\begin{figure}[ptb]
\includegraphics[width=90mm]{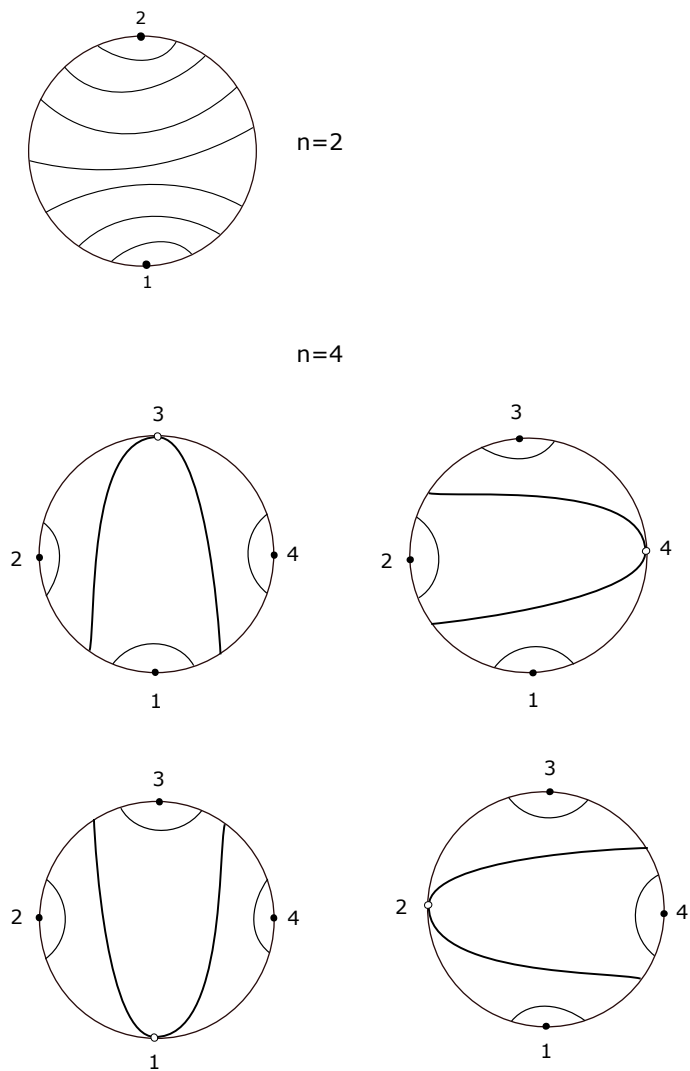}\caption{$T_{2}=1$, $T_{4}=4$.}%
\label{f2}%
\end{figure}

\begin{figure}[ptb]
\includegraphics[width=110mm]{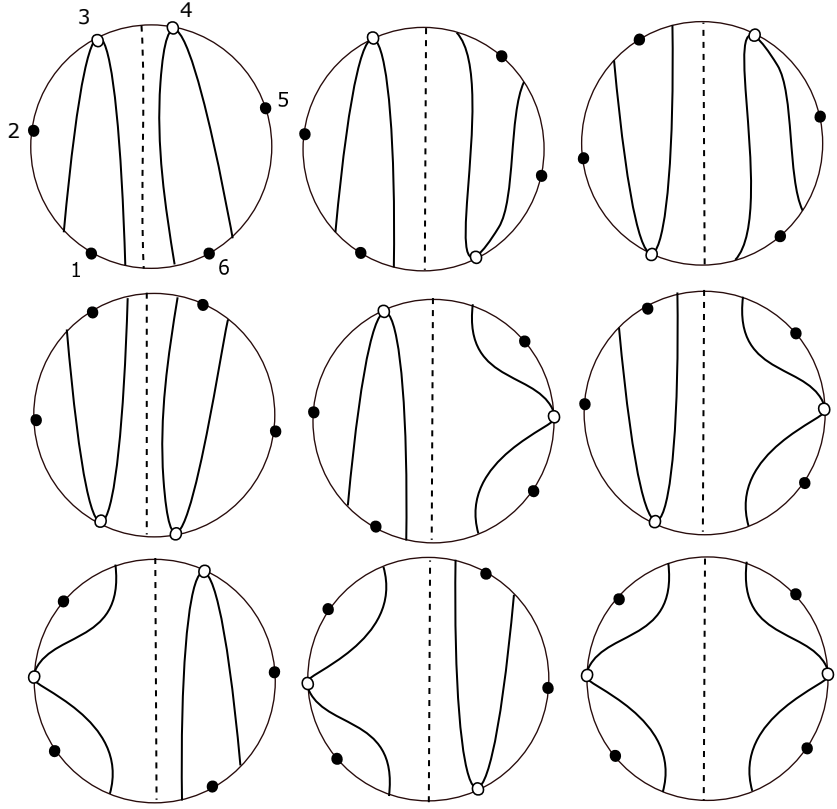}\caption{$T_{6}=3.9=27$ (rotation at
$\frac{\pi}{3}$\ and $\frac{2\pi}{3}$).}%
\label{f3}%
\end{figure}
\end{center}

so, $T_{2}=1$, $T_{4}=4$, $T_{6}=27$, $T_{8}=216$, etc. A quick check in
OEIS's database suggests that highly likely $\left\{  T_{2k}\right\}  $ is in
fact sequence A275607 (see \cite{b4}). We show this later by comparing formula
(\ref{0}) with the explicit expression defining A275607. The diagrams
adjusting that $T_{2}=1$, $T_{4}=4$ and $T_{6}=27$ are depicted at
\hyperref[f2]{Fig.~\ref*{f2}} and \hyperref[f3]{Fig.~\ref*{f3}}, respectively.
In the latter figure one has to perform rotations at angle $\frac{\pi}{3}%
$\ and $\frac{2\pi}{3}$ to get all $3.9=27$ solutions.\ All 27 diagrams are
different, which becomes evident after examinating the position of the
vertical punctured line after rotation.

Note that the numbers $T_{n}$ are counting vector fields only
\textit{combinatorially}, not \textit{topologically}. The latter problem seems
to be much harder and will be discussed at the end.

\section{In identity with Catalan numbers}

The Catalan numbers may be defined by the basic recurrence relation%
\begin{equation}
C_{0}=1,\ \ \ \ C_{n+1}=\sum\limits_{i=0}^{n}C_{i}C_{n-i}. \label{2}%
\end{equation}

Iterating the above identity, it is not difficult to obtain some relation
involving triple products $C_{i}C_{j}C_{k}$.

\begin{proposition}
\label{prop}For $n\geq2$ the following identity holds true%
\begin{equation}
\sum\limits_{i+j+k=n-1}C_{i}C_{j}C_{k}=C_{n+1}-C_{n}, \label{3}%
\end{equation}

where the sum is taken over all triples $(i,j,k)$ such that $i+j+k=n-1$ and
$i,j,k\geq0$.
\end{proposition}

For example, for $n=3$ the above formula gives%
\[
C_{0}C_{1}C_{1}+C_{1}C_{0}C_{1}+C_{1}C_{1}C_{0}+C_{0}C_{0}C_{2}+C_{0}%
C_{2}C_{0}+C_{2}C_{0}C_{0}=C_{4}-C_{3}.
\]

After substitution we get%
\[
1+1+1+2+2+2=14-5,
\]

which is, of course, true.

To prove the general case, it suffices to write the basic relation (\ref{2})
in the form%
\[
C_{n+1}-C_{n}=\sum\limits_{i=1}^{n}C_{i}C_{n-i}=\sum\limits_{i=1}^{n}\left(
\sum\limits_{j=0}^{i-1}C_{j}C_{i-j-1}\right)  C_{n-i}%
\]

and to notice that the last expression is exactly the left-hand side of
(\ref{3}).

\section{Proof of the main formula}

Let $D_{n}^{0}\subset D_{n}$ be the set of ($n$-nodes)\ diagrams with a fixed
node (say $p_{1}$) which is prescribed to be an $a$-node. The following
proposition is counting the number of such field classes.

\begin{lemma}
\label{lemma}$\left\vert D_{n}^{0}\right\vert =3^{\frac{n}{2}-1}C_{\frac{n}%
{2}-1}=3^{\frac{n}{2}-1}\frac{(n-2)!}{\left(  \frac{n}{2}\right)  !\left(
\frac{n}{2}-1\right)  !}$.
\end{lemma}

\begin{proof}
Set $\alpha_{n}=\left\vert D_{n}^{0}\right\vert $. Note that there are exactly
3 situations realizing elements from $D_{n}^{0}$, these are depicted at
\hyperref[f6]{Fig.~\ref*{f6}}. This may be seen by following the flow lines
going out of node $p_{1}$ and notice that they have to lean on a touching line
$\lambda$ (for $n\geq4$). Then, according to the position of this line with
respect to $p_{1}$, there are 3 topologically different possibilities
represented at \hyperref[f6]{Fig.~\ref*{f6}}. The above touching line is
separating the disk into 3 regions, say $A,B,C$, we shall suppose that $C$ is
containing node $p_{1}$. Note that in region $C$ the flow is very simple and
topologically unique (a sort of a flow-box), so $C$ will not take part in the calculation.

Observe furthermore that the restriction of the field diagram under
consideration on regions $A$ and $B$ may be interpreted as some elements from
$D_{k}^{0}$ and $D_{n-k}^{0}$, respectively, where $k$ is even and $2\leq
k\leq n-2$. This may be seen as follows: we may split $\mathbb{B}^{2}$ along
the touching line $\lambda$ into 3 disjoint pieces and then to shrink the part
of $\lambda$ lying on each one, as shown at \hyperref[f7]{Fig.~\ref*{f7}}. In
such a way we get 3 smaller disks $A_{0},B_{0},C_{0}$ with a new born node on
each one that occurs to be an $a$-node (source). This node will be set to be
the distinguished one from the definition of $D_{n}^{0}$.\ So, the induced
fields on $A_{0},B_{0},C_{0}$ are in fact elements from $D_{k}^{0}$,
$D_{n-k}^{0}$ and $D_{2}^{0}$, respectively. Of course, $D_{2}^{0}$ is a
1-point set, so it doesn't take part in the calculation.

\begin{figure}[ptb]
\includegraphics[width=100mm]{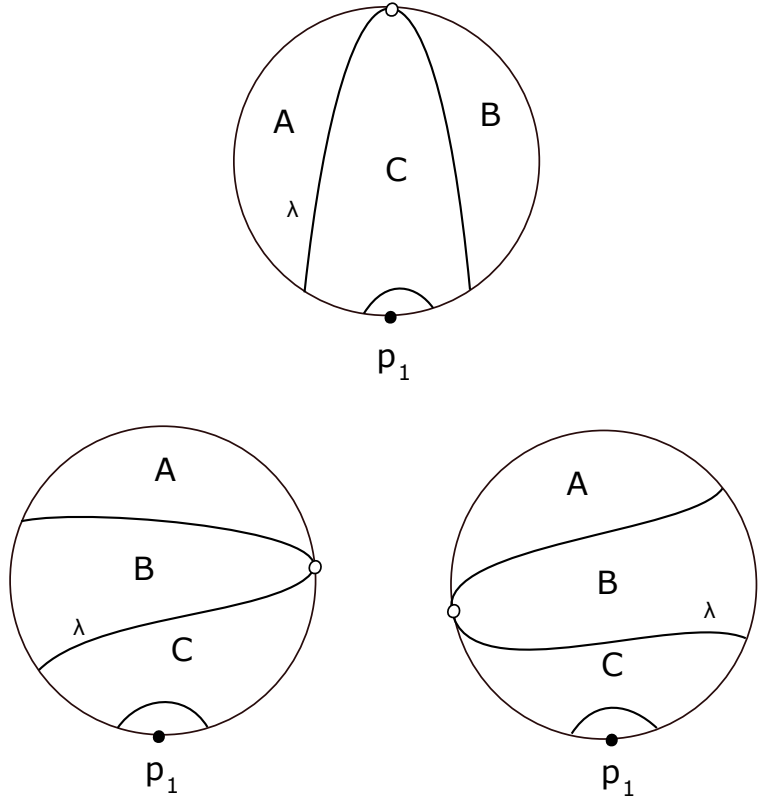}\caption{3 cases for regions $A,B,C$,
the flow on $C$ is trivial.}%
\label{f6}%
\end{figure}

\begin{figure}[ptb]
\includegraphics[width=90mm]{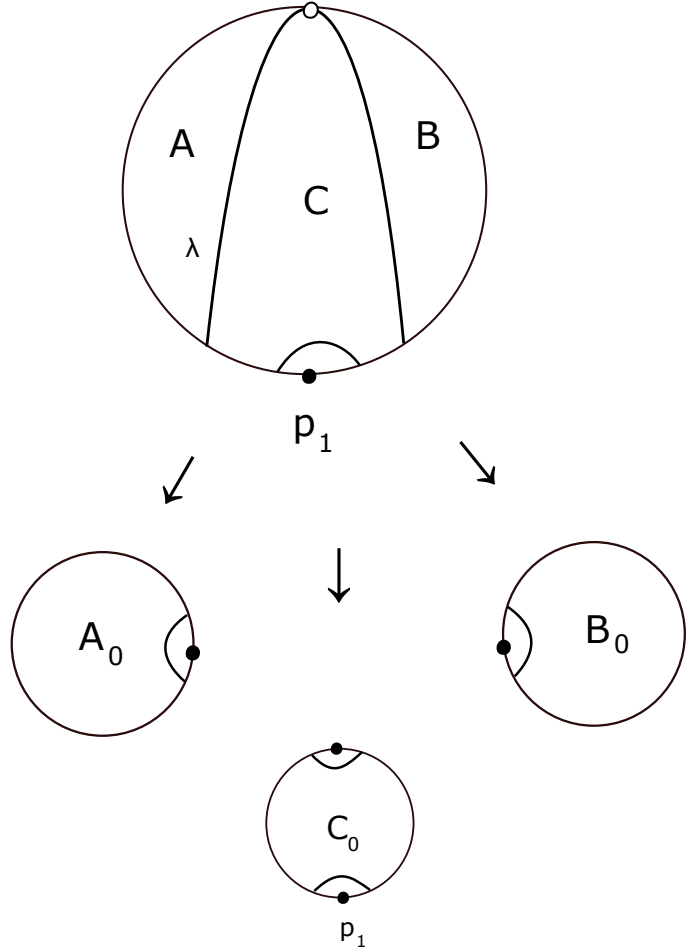}\caption{Splitting the disk
$\mathbb{B}^{2}$ along a touching line $\lambda$.}%
\label{f7}%
\end{figure}

\begin{figure}[ptb]
\includegraphics[width=60mm]{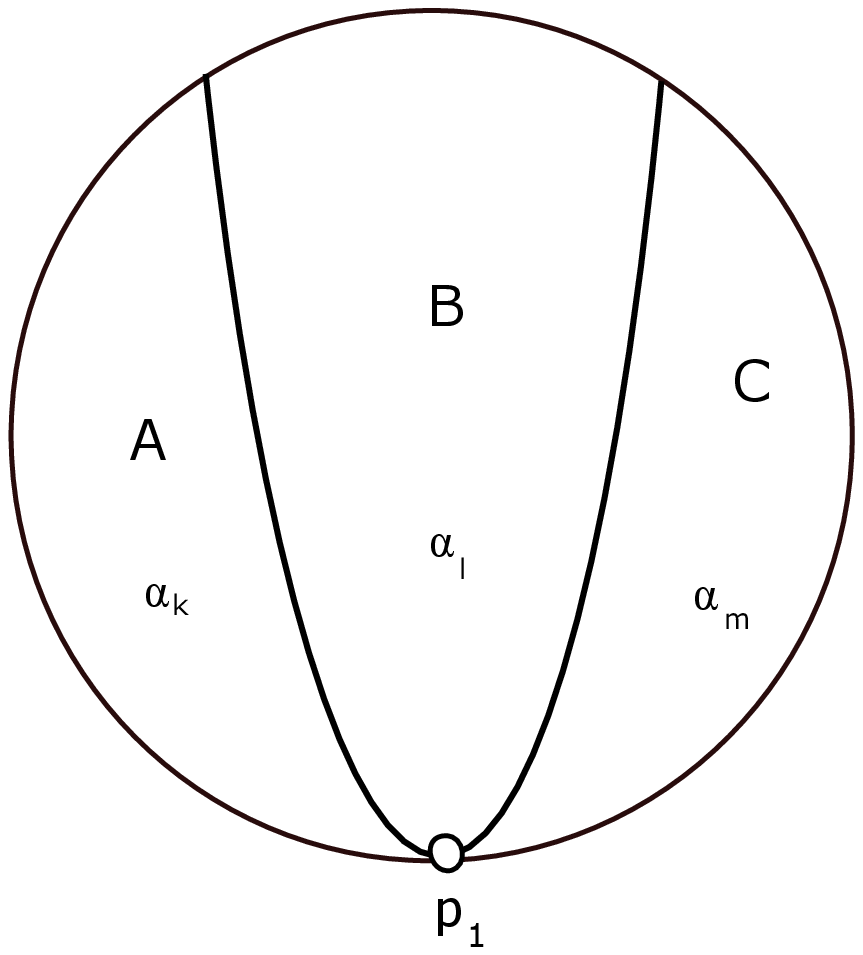}\caption{The case when $p_{1}$ is a
$b$-node.}%
\label{f8}%
\end{figure}

The next observation is that the 3 cases from \hyperref[f6]{Fig.~\ref*{f6}}
are identical from computational point of view. Therefore for numbers
$\alpha_{n}=\left\vert D_{n}^{0}\right\vert $ we get the following recurrence
relation%
\[
\alpha_{n}=3\left(  \alpha_{2}\alpha_{n-2}+\alpha_{4}\alpha_{n-4}+\dots
+\alpha_{n-4}\alpha_{4}+\alpha_{n-2}\alpha_{2}\right)  .
\]

Now it is not difficult to prove the lemma by induction. Suppose the formula
is true for each $k\leq n-2$. Clearly, for $n=2$ it is true, as $\alpha_{2}%
=1$. Then, after substitution in the right-hand side of the recurrence
relation, one gets
\[
\alpha_{n}=3^{\frac{n}{2}-1}\left(  C_{0}C_{\frac{n}{2}-2}+C_{1}C_{\frac{n}%
{2}-3}+\dots+C_{\frac{n}{2}-3}C_{1}+C_{\frac{n}{2}-2}C_{0}\right)
=3^{\frac{n}{2}-1}C_{\frac{n}{2}-1}\text{,}%
\]

where in the last equality we make use of the basic relation (\ref{2}) about
Catalan numbers.

The lemma is proved.
\end{proof}

Note that $\alpha_{n}$ is in fact sequence\ A005159 from OEIS.

We are now ready to prove the main formula about arbitrary diagrams.

\begin{theorem}
Let $T_{n}=\left\vert D_{n}\right\vert .$ Then $T_{n}=3^{\frac{n}{2}-2}\left(
C_{\frac{n}{2}}+2C_{\frac{n}{2}-1}\right)  $.
\end{theorem}

\begin{proof}
Take some $v\in D_{n}$ and consider node $p_{1}$. There are two cases:

1) $p_{1}$ is an $a$-node

2) $p_{1}$ is a $b$-node.

Clearly, $T_{n}$ is a sum of cases 1) and 2). In case 1) the number of such
fields equals $\alpha_{n}$ and is counted by \hyperref[lemma]%
{Lemma~\ref*{lemma}}.

Suppose now that $p_{1}$ is a $b$-node and $\lambda$ is the touching line
passing through it (see \hyperref[f8]{Fig.~\ref*{f8}}). Then $\lambda$ is
separating the disk into some regions $A,B,C$. Now, splitting $\mathbb{B}^{2}$
along $\lambda$ and reasoning exactly as in the proof of \hyperref[lemma]%
{Lemma~\ref*{lemma}}, it is clear that for this particular choice of the
touching line $\lambda$ we get a summand in case 2) of the form $\alpha
_{k}\alpha_{l}\alpha_{m}$ where $k+l+m=n+2$ for some even numbers $k,l,m\geq
2$. Taking into account these remarks we have%
\[
T_{n}=\alpha_{n}+\sum\limits_{k+l+m=n+2}\alpha_{k}\alpha_{l}\alpha_{m}\text{,}%
\]

where the sum runs over all triples of even numbers $k,l,m\geq2$ for which
$k+l+m=n+2$. We now refer to \hyperref[lemma]{Lemma~\ref*{lemma}} to calculate%
\[
\alpha_{k}\alpha_{l}\alpha_{m}=3^{\frac{k}{2}-1}3^{\frac{l}{2}-1}3^{\frac
{m}{2}-1}C_{\frac{k}{2}-1}C_{\frac{l}{2}-1}C_{\frac{m}{2}-1}=3^{\frac{n}{2}%
-2}C_{\frac{k}{2}-1}C_{\frac{l}{2}-1}C_{\frac{m}{2}-1}\text{.}%
\]

Furthermore, it is not difficult to see that by \hyperref[prop]%
{Proposition~\ref*{prop}}
\[
\sum\limits_{k+l+m=n+2}C_{\frac{k}{2}-1}C_{\frac{l}{2}-1}C_{\frac{m}{2}%
-1}=C_{\frac{n}{2}}-C_{\frac{n}{2}-1},
\]

so, finally we have%
\[
T_{n}=3^{\frac{n}{2}-1}C_{\frac{n}{2}-1}+3^{\frac{n}{2}-2}\left(  C_{\frac
{n}{2}}-C_{\frac{n}{2}-1}\right)  ,
\]

and after simplification one gets%
\[
T_{n}=3^{\frac{n}{2}-2}\left(  C_{\frac{n}{2}}+2C_{\frac{n}{2}-1}\right)  .
\]

The theorem is proved.
\end{proof}

\begin{remark}
It follows from the proof of the theorem that each touching line $\lambda$
takes part into exactly $3^{\frac{n}{2}-2}C_{\frac{k}{2}-1}C_{\frac{l}{2}%
-1}C_{\frac{m}{2}-1}$ diagrams from $D_{n}$. Here $k,l,m$ depend on the size
of regions $A,B,C$.
\end{remark}

\begin{remark}
It might be of some interest to consider the following specification of the problem:

Find the number of diagrams from class $D_{n}$ with prescribed distribution of
\newline$a$- and $b$-nodes. Clearly, it suffices to fix the places of
$b$-nodes which may be done by $\binom{n}{n/2-1}$ ways. Note that the result
sensitively depends on the initial distribution of the nodes. For example, if
$n=6$ we have to choose a pair of $b$-nodes and one checks that there are 3
possible results for the corresponding number: $(1,2)\rightarrow1$,
$(1,3)\rightarrow2$, $(1,4)\rightarrow3$. For greater values of $n$ the
problem seems much harder.
\end{remark}

\section{Sequence A275607}

As we mentioned at the beginning, the calculation of the first terms of
sequence $T_{2k}$, $k=1,2\dots$ gives%
\[
\mathbf{1,4,27,216,1890,17496,168399,\dots}\text{,}%
\]

which coincide with the first terms of sequence A275607 from OEIS. An explicit
formula for A275607 was found by Charles R. Greathouse (see \cite{b4})%
\[
a(n)=\frac{2.3^{n}(n+1)}{2n^{2}+n}\binom{2n+1}{n-1}.
\]

It may be shown that $T_{2n+2}=a(n)$, or%
\[
3^{n-1}\left(  C_{n+1}+2C_{n}\right)  =\frac{2.3^{n}(n+1)}{2n^{2}+n}%
\binom{2n+1}{n-1},
\]

identifying in such a way our sequence $T_{2k}$ with A275607. By the way,
Wolfram Alpha gives the following equality%
\[
3^{n-1}\left(  C_{n+1}+2C_{n}\right)  =\frac{2.4^{n}(n+1)\Gamma(n+\frac{1}%
{2})}{\sqrt{\pi}\Gamma(n+3)},
\]

and since the right-hand side is the original definition of A275607, this is
another way to identify $T_{2k}$.

Taking into account that the (ordinary) generating function for Catalan
numbers is $\frac{1-\sqrt{1-4x}}{2x}$, it is not difficult to find a
generating function for sequence $T_{2k}$:%
\[
T(x)=\frac{(1+6x)(1-\sqrt{1-12x})}{54x}.
\]

Indeed, a calculation with Wolfram Alpha gives%
\[
\frac{(1+6x)(1-\sqrt{1-12x})}{54x}=\frac{1}{9}+x+4x^{2}+27x^{3}+216x^{4}%
+1890x^{5}+17496x^{6}+\dots\text{,}%
\]

so the coefficient of $x^{k}$ equals $T_{2k}$.

\section{How to find all diagrams with $n$ nodes}

The knowledge of numbers $T_{n}$ doesn't tell us how to find the corresponding
diagrams with $n$ nodes. Even for $n=8$ the attempt to find these \textit{by
hand} seems counterproductive, as their number is 216.

Here we propose an algorithm for this task, which has two steps:

1) constructing some graph $\Gamma_{n}$ whose maximal independent vertex sets
are in 1-to-1 correspondence with the $n$ nodes diagrams

2) finding all such independent sets in $\Gamma_{n}$ and then tracing back all
diagrams with $n$ nodes $D_{n}$.

The productivity of this algorithm depends on the algorithm selected for
enumerating the maximal independent sets in $\Gamma_{n}$. Note that there
might be a significant improvement of the known general algorithms, in view of
the fact that all the maximal independent sets of $\Gamma_{n}$ surely have one
and the same size = $\frac{n}{2}-1$. Moreover, the number of all such sets is
known \textit{a priori} and should equal number $T_{n}$. Another advantage is
the observation that the graph $\Gamma_{n}$ has a rich group of symmetries, so
each maximal independent set automatically produces a series of such sets in
$\Gamma_{n}$.

\textbf{1.} \textbf{Construction of graph }$\Gamma_{n}$

\textbf{a) The vertices}

Roughly speaking, the vertices of $\Gamma_{n}$ are the virtual touching lines
$\lambda$ in the disk and two $\lambda,\lambda^{\prime}$ are connected by an
edge if and only if they are, in some sense, linked in $\mathbb{B}^{2}$.

Suppose for simplicity that the set of nodes is a regular $n$-gon and the
nodes are clockwise marked by integers%
\[
A=\left\{  1,2,\dots,n\right\}  .
\]
Then the vertices of $\Gamma_{n}$ are all triples of nodes $(a,p,q)$ such that

1) $a,p,q\in A$, $a\neq p,q$ and all 3 nodes have the same parity, e.g.
$(1,3,7)$, $(8,2,6)$

2) the triple $(a,p,q)$ is clockwise oriented on $\mathbb{S}^{1}$, in case
$a,p,q$ are different

3) all triples of type $(a,p,p)$, where $a$ and $p$ are different and have the
same parity, belong to $\Gamma_{n}$, e.g. $(2,4,4)$, $(4,2,2)$, $(7,3,3)$.

It turns out that the touching lines of a given diagram from $D_{n}$ may be
coded by such triples as follows:

- $a$ is the node where line $\lambda$ is touching $\mathbb{S}^{1}$

- $p$ and $q$ are the nodes closest to $\lambda$ which lie in the component of
$\mathbb{B}^{2}\backslash\lambda$ containing the whole $\lambda$ in its
boundary (see \hyperref[f9]{Fig.\ref*{f9}-1}).

Note that the case $(a,p,p)$ is not excluded and corresponds to the situation
from \hyperref[f9]{Fig.\ref*{f9}-2}. Another important observation is that the
parity of $a,p,q$ is the same, as there is an odd number of nodes between $a$
and either of $p,q$. In such a way, any touching line is uniquely identified
by some triple $(a,p,q)\in\Gamma_{n}$.

\begin{figure}[ptb]
\includegraphics[width=95mm]{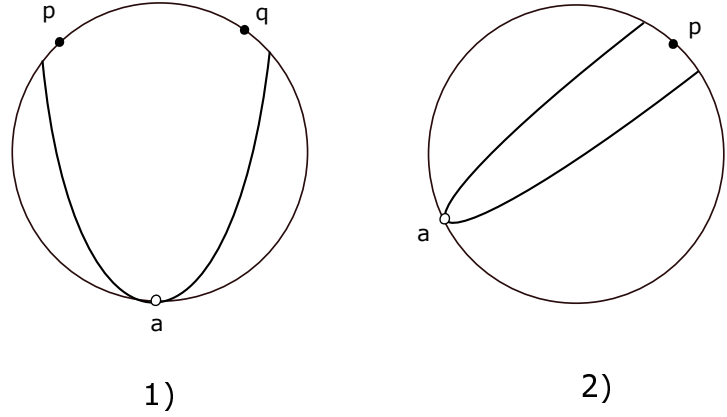}\caption{Vertices of $\Gamma_{n}$ of
type $(a,p,q)$ and $(a,p,p)$.}%
\label{f9}%
\end{figure}

Furthermore, it is not difficult to calculate the number of vertices of
$\Gamma_{n}$:%
\[
\left\vert \Gamma_{n}\right\vert =\frac{n^{2}(n-2)}{8}\text{.}%
\]

For example $\left\vert \Gamma_{4}\right\vert =4$, $\left\vert \Gamma
_{6}\right\vert =18$, that may be checked by hand.

\textbf{b) The edges}

Now we have to define when two vertices of $\Gamma_{n}$ are connected by an
edge. We shall proceed in an equivalent way by defining when two vertices
\textit{are not} connected by an edge. So, we shall define the complementary
graph $\overline{\Gamma_{n}}$ rather than $\Gamma_{n}$. The cause for this
approach is that graph $\Gamma_{n}$ turns out to be quite \textit{dense}, so
$\overline{\Gamma_{n}}$ should be \textit{sparse} and thus easier to describe
(depict). Then, of course, each invariant of $\overline{\Gamma_{n}}$ transfers
to a dual invariant of $\Gamma_{n}$ and vice versa.

It is convenient to consider addition in the circular set $A$ modulo $n$, e.g.
$n+1=1$.

Hereafter we describe the situations when two vertices of $\overline
{\Gamma_{n}}$ are connected by an edge. Let $v_{1}=(a_{1},p_{1},q_{1}%
),v_{2}=(a_{2},p_{2},q_{2})$ be two vertices of $\overline{\Gamma_{n}}$.

\ 1a) If $a_{1}=a_{2}$, then $v_{1}$ and $v_{2}$ are \textit{not }connected in
$\overline{\Gamma_{n}}$.

\ 1b) $(a,b,b)$ and $(b,a,a)$ are also not connected in $\overline{\Gamma_{n}%
}$ (see \hyperref[f10]{Fig.\ref*{f10}-2})

2) If the nodes $a_{i},p_{i},q_{i}$ are all different, then

\ \ 2a) $v_{1}$ and $v_{2}$ are connected if $Co(v_{1})\cap Co(v_{2}%
)=\varnothing$, where $Co(v)$ is the convex hull of the nodes of $v$, see
\hyperref[f11]{Fig.\ref*{f11}-1}.

\ \ 2b) In case $Co(v_{1})\cap Co(v_{2})\neq\varnothing$, then $v_{1}$ and
$v_{2}$ are not connected in $\overline{\Gamma_{n}}$ (and thus connected in
$\Gamma_{n}$, see \hyperref[f10]{Fig.\ref*{f10}-1}).

E.g. $(1,3,5)$ and $(8,6,7)$ are connected, while $(1,3,5)$ and $(8,2,4)$ are not.

3) The following pairs are connected:

\ \ 3a) $(a,p+1,q)$ and $(b,q+1,p)$, where $a$ and $b$ have different parity,
all 6 nodes involved are different and segment $[a,b]$ is separating $[p,p+1]$
from $[q,q+1]$. E.g. $(1,3,7)$ and $(6,8,2)$. This situation is depicted on
\hyperref[f11]{Fig.\ref*{f11}-2}

\ \ 3b) $(a,p,b)$ and $(b,p,x)$ where $a$ and $b$ have the same parity, e.g.
$(2,4,8)$ and $(8,4,6)$, $(1,3,5)$ and $(5,3,3)$; see \hyperref[f12]%
{Fig.\ref*{f12}-1}

\ \ 3c) $(a,b,p)$ and $(b,x,p)$ where $a$ and $b$ have the same parity, e.g.
$(2,6,8)$ and $(6,8,8)$; see \hyperref[f12]{Fig.\ref*{f12}-2}.

\begin{figure}[ptb]
\includegraphics[width=95mm]{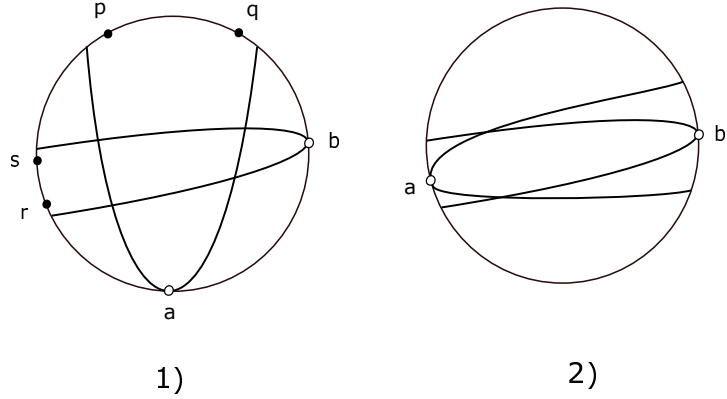}\caption{Edges in $\Gamma_{n}$.}%
\label{f10}%
\end{figure}

\begin{center}
\begin{figure}[ptb]
\includegraphics[width=95mm]{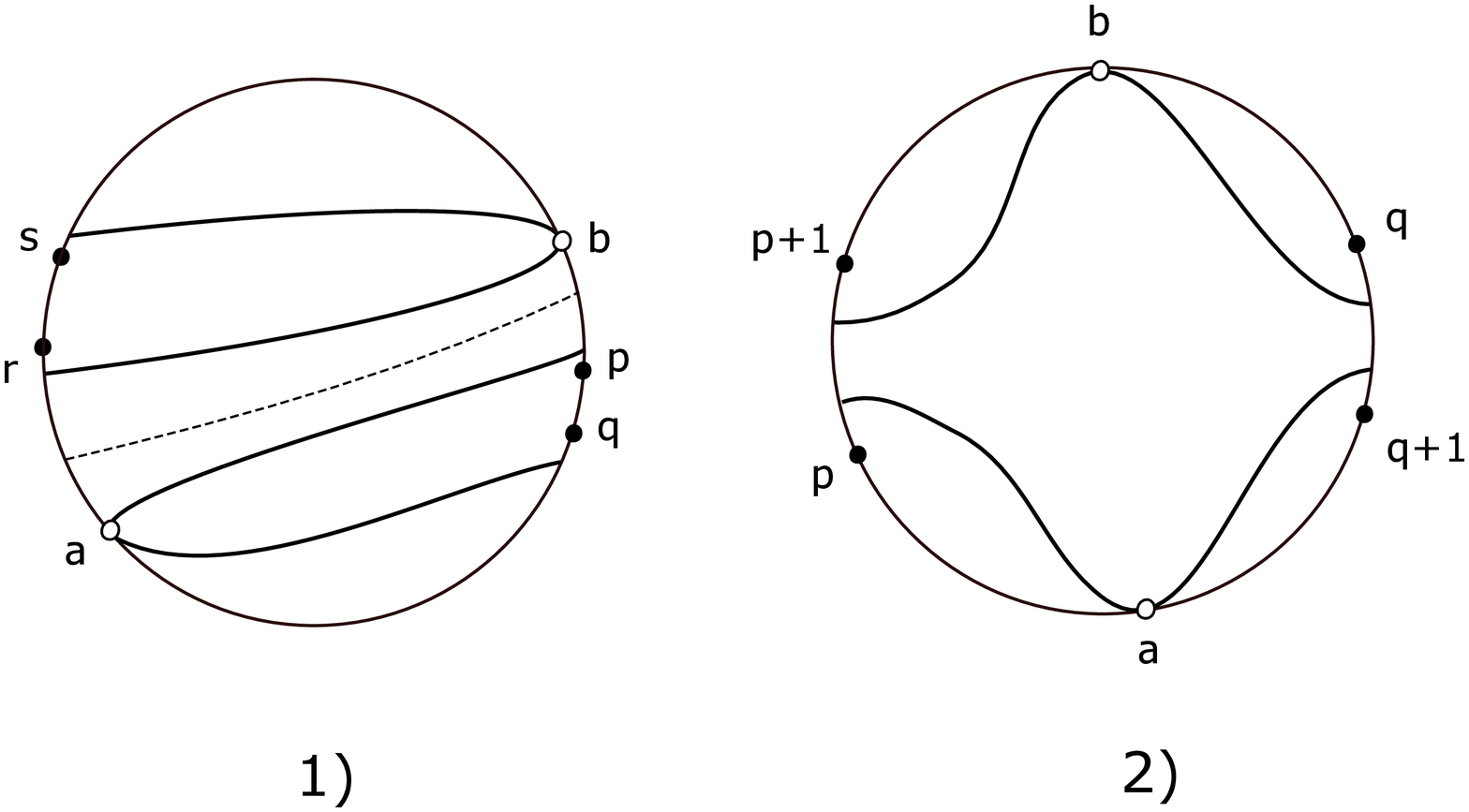}\caption{Edges in $\overline
{\Gamma_{n}}$.}%
\label{f11}%
\end{figure}

\begin{figure}[ptb]
\includegraphics[width=95mm]{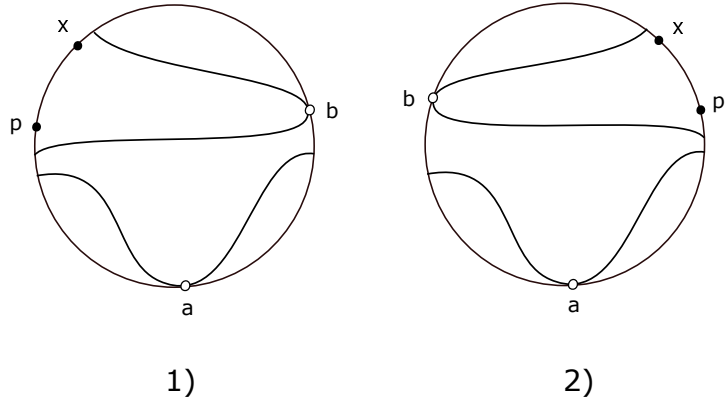}\caption{Edges in $\overline
{\Gamma_{n}}$.}%
\label{f12}%
\end{figure}
\end{center}

\ Clearly, graph $\Gamma_{4}$ is a 4-clique, while $\Gamma_{6}$ may be
characterized by the observation that its complementary graph $\overline
{\Gamma_{6}}$ is a disjoint sum of 3 copies of the Kuratowski graph $K_{3,3}$.
This becomes evident after examination of \hyperref[f3]{Fig.~\ref*{f3}}. Note that $\overline{\Gamma
_{6}}=K_{3,3}\sqcup K_{3,3}\sqcup K_{3,3}$ is triangle free and has 27 edges,
which illustrates \hyperref[t1]{Theorem~\ref*{t1}} , since $T_{6}=27$.

It is not difficult to show that $\Gamma_{n}$ is connected for $n\geq6$, while
$\overline{\Gamma_{n}}$ is connected for $n\geq8$. Also, only $\Gamma_{4}$ and
$\Gamma_{6}$ are regular, the others are not. For example, $\Gamma_{8}$ has
two types of vertices - of degree 6 and 15.

The importance of graph $\Gamma_{n}$ rely on the observation that each diagram
from $D_{n}$ corresponds to a maximal independent set of vertices of
$\Gamma_{n}$ and vice versa. Then, as explained above, it is convenient to
pass to the complementary graph $\overline{\Gamma_{n}}$ and to notice that the
maximal independent sets in $\Gamma_{n}$ correspond to maximal cliques in
$\overline{\Gamma_{n}}$.

\begin{theorem}
\label{t1}The maximal cliques of graph $\overline{\Gamma_{n}}$ are in 1-to-1
correspondence with the diagrams from $D_{n}$. All the maximal cliques have
one and the same size = $\frac{n}{2}-1$. Their number is $T_{n}=3^{\frac{n}%
{2}-2}\left(  C_{\frac{n}{2}}+2C_{\frac{n}{2}-1}\right)  $, where $C_{m}$ is
the $m$-th Catalan number.
\end{theorem}

The proof follows easily from the above considerations.

Then one may refer to the known algorithms for finding the maximal cliques of
a given graph, e.g. the Bron-Kerbosch algorithm (see \cite{b5}). However, we
wouldn't recommend the general algorithms in our case. For example, the
worst-case running time of the Bron-Kerbosch algorithm is $O\left(
3^{\frac{N}{3}}\right)  $, where $N$ is the number of vertices, which applied
to $\overline{\Gamma_{n}}$ gives $O\left(  3^{\frac{n^{3}}{24}}\right)  $ as
an a priori estimate (very slow). Hereafter we shall list the most important
properties of $\Gamma_{n}$ and $\overline{\Gamma_{n}}$ that might suggest some
much better algorithm.

1. The graphs $\Gamma_{n}$ and $\overline{\Gamma_{n}}$ are connected for
$n\geq8$.

2. $\left\vert \Gamma_{n}\right\vert =\left\vert \overline{\Gamma_{n}%
}\right\vert =\frac{n^{2}(n-2)}{8}$.

3. All the maximal cliques of $\overline{\Gamma_{n}}$ have one and the same
size = $\frac{n}{2}-1$.

4. The number of all maximal cliques of $\overline{\Gamma_{n}}$ is
$3^{\frac{n}{2}-2}\left(  C_{\frac{n}{2}}+2C_{\frac{n}{2}-1}\right)  $.

5. Each vertex of $\overline{\Gamma_{n}}$ takes part into exactly $3^{\frac
{n}{2}-2}C_{\frac{k}{2}-1}C_{\frac{l}{2}-1}C_{\frac{m}{2}-1}$ maximal cliques,
where $k,l,m$ depend on the geometric line representing the vertex.

6. $\Gamma_{n}$ and $\overline{\Gamma_{n}}$ are not regular for $n\geq8$, but
the degree of each vertex may be calculated considering the position of the
touching line corresponding to the vertex.

7. The group $G_{n}=\mathbb{Z}_{n}\oplus\mathbb{Z}_{2}^{n}$ is acting in a
natural way in both $\Gamma_{n}$ and $\overline{\Gamma_{n}}$. This is
explained in the next section.

\section{Counting fields topologically}

We shall try in this final section to settle the general problem for counting
stationary points free vector fields in $\mathbb{B}^{2}$ from topological
point of view. Note that in our previous counting scheme one and the same
topological picture appears many times, see for example \hyperref[f3]{Fig.~\ref*{f3}}. So, it is a good
point to eliminate multiple pictures and then to count topologically different
diagrams. Let us say from the beginning that we won't do this here. This
problem seems to be much harder than the combinatorial one solved in
Section~3. Anyway, we shall describe formally the topology equivalence classes
of fields, reducing these to a quite concrete object. The problem is to find
its cardinality.

It turns out that the answer to the question is closely related to the action
of group $G_{n}=\mathbb{Z}_{n}\oplus\mathbb{Z}_{2}^{n}$ on $\overline
{\Gamma_{n}}$.

We say that group $G$ is acting on graph $\Gamma$, if $G$ is acting on the
vertex set of $\Gamma$ and is incidence preserving. It is geometrically
evident that $\mathbb{Z}_{n}$ is acting in $\Gamma_{n}$ and $\overline
{\Gamma_{n}}$ by rotations in the set of nodes at angle $2\pi/n$, while
$\mathbb{Z}_{2}^{n}$ is acting by symmetries about the lines containing pairs
of opposite nodes $\left(  k,k+\frac{n}{2}\right)  $, as well as pairs of type
$\left(  k+\frac{1}{2},k+\frac{n+1}{2}\right)  $\ (recall that addition is
$\operatorname{mod}n$). Let now $\Delta_{n}$ be the set of all maximal cliques
in $\overline{\Gamma_{n}}$, then the action of $G_{n}$ on $\overline
{\Gamma_{n}}$ transfers to an action on $\Delta_{n}$.\ Note that this action
is \textit{not} \textit{free} and even has fixed points. Each element of
$\Delta_{n}$ is periodic and its period is a divisor of $2n$. The action of
the $\mathbb{Z}_{n}$-factor of $G_{n}$ cannot have involutory or fixed
elements (for $n\geq4$), while, of course, each element of $\Delta_{n}$ is
involutory under $\mathbb{Z}_{2}^{n}$. One advantage, we may take profit of,
is the following simple observation:

If $C\in\Delta_{n}$ is a maximal clique of $\overline{\Gamma_{n}}$, then all
the sets%
\[
GC=\left\{  gC~|~g\in G_{n}\right\}
\]

are maximal cliques as well. This may provide us with many different diagrams
from $D_{n}$ for free, in the best case with $2n$ new diagrams.

Topological equivalence of stationary points free vector fields in
$\mathbb{B}^{2}$ is defined in a natural way:

$v_{1}\sim v_{2}$ if there is a diffeomorphism $\varphi:\mathbb{B}%
^{2}\rightarrow\mathbb{B}^{2}$ which is transforming the flow lines of $v_{1}$
onto those of $v_{2}$. (Here $\varphi$ is allowed to be \textit{orientation
reversing}.)

Note that if $v_{1}\sim v_{2}$, then surely $v_{1}$ and $v_{2}$ both
have$\ \frac{n}{2}+1$ $a$-nodes and $\frac{n}{2}-1$ $b$-nodes on
$\mathbb{S}^{1}$ for some even $n$. Let $R_{n}$ denote the set of equivalence
classes of such fields. Clearly, the problem is to find its cardinality
$\left\vert R_{n}\right\vert $.

\begin{proposition}
Let $\Delta_{n}$ be the set of maximal cliques of graph $\overline{\Gamma_{n}%
}$ and the group $G_{n}=\mathbb{Z}_{n}\oplus\mathbb{Z}_{2}^{n}$ is acting on
it as described. Let $\Delta_{n}|_{G_{n}}$ be the orbit space under this
action. Then%
\[
\left\vert \Delta_{n}|_{G_{n}}\right\vert =\left\vert R_{n}\right\vert .
\]

\end{proposition}

\begin{remark}
If in the definition of equivalent fields we restrict ourselves to orientation
preserving diffeomorphisms, it suffices to consider only the $\mathbb{Z}_{n}%
$-action on $\Delta_{n}$. Namely, if $R_{n}^{0}\subset R_{n}$ are the
equivalence classes defined by such diffeomorphisms, then%
\[
\left\vert \Delta_{n}|_{\mathbb{Z}_{n}}\right\vert =\left\vert R_{n}%
^{0}\right\vert .
\]

\end{remark}

The above proposition provides us with some \textquotedblleft brute
force\textquotedblright\ algorithm for finding $\left\vert R_{n}\right\vert $
and $\left\vert R_{n}^{0}\right\vert $.

Obviously, $\left\vert R_{4}\right\vert =\left\vert R_{4}^{0}\right\vert =1$.
For $n=6$ things are equally easy; see \hyperref[f14]{Fig.\ref*{f14}} to
justify that $\left\vert R_{6}\right\vert =4$ , $\left\vert R_{6}%
^{0}\right\vert =6$.

\begin{figure}[ptb]
\includegraphics[width=120mm]{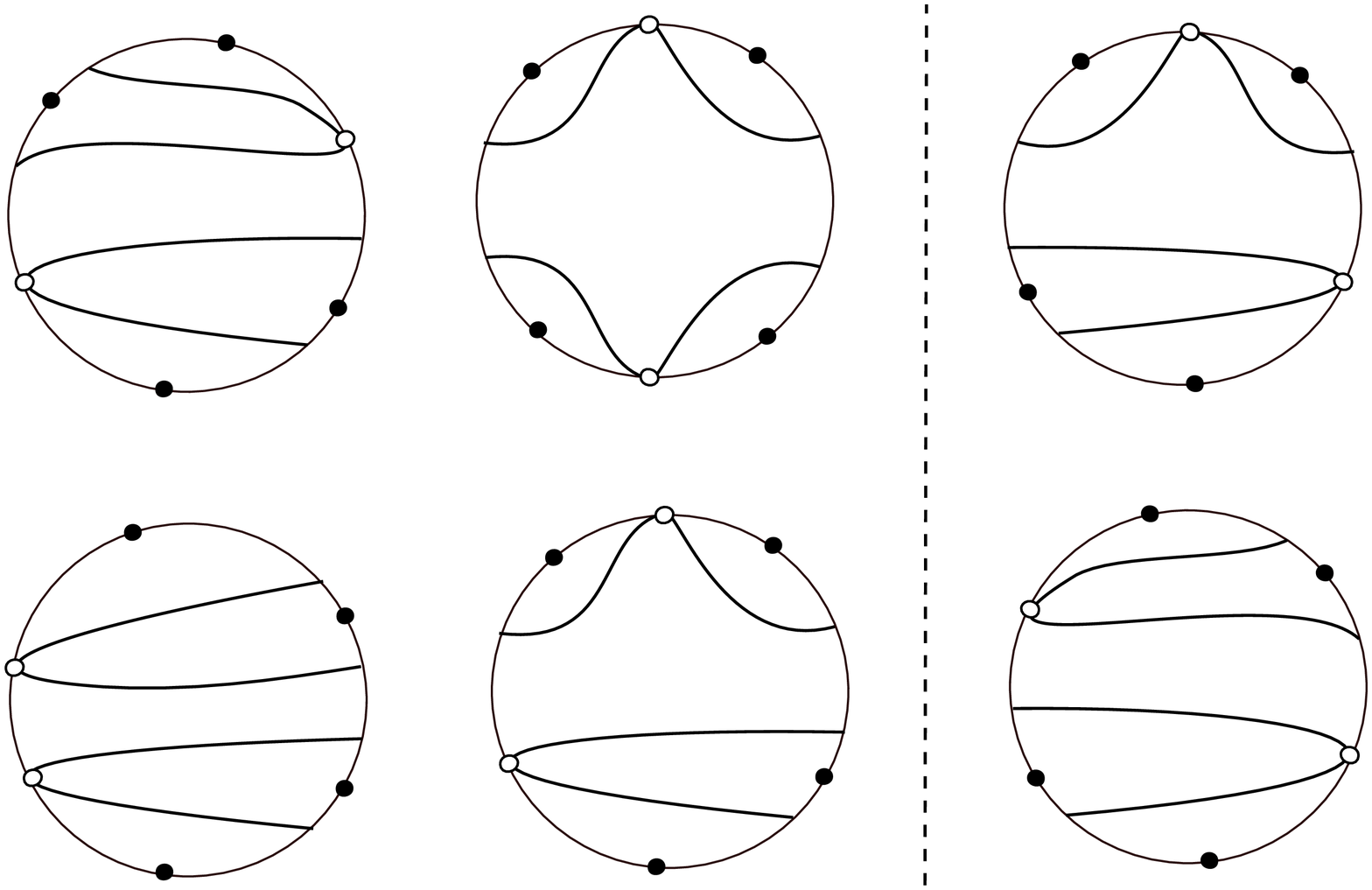}\caption{$\left\vert R_{6}\right\vert
=4$ , $\left\vert R_{6}^{0}\right\vert =6$.}%
\label{f14}%
\end{figure}

It is clear that the above problem has a local variant:

Given some maximal clique $C\in\Delta_{n}$ (= diagram from $D_{n}$), then what
is the cardinality of its orbit $\left\{  gC~|~g\in G_{n}\right\}  $?

\section{Final remarks}

Let us finally list some natural problems, that might trace a route for
further investigations.

1) Find the number of topological diagrams $\left\vert R_{n}\right\vert $ and
$\left\vert R_{n}^{0}\right\vert $ (see Section 6). Find an algorithm for
generating these diagrams.

2) Find some \textquotedblleft good\textquotedblright\ algorithm for
generating diagrams from $D_{n}$ based on the high symmetry of graph
$\overline{\Gamma_{n}}$.

3) Enlarge the class of vector fields under consideration by admitting
multiple touches between a flow line and the boundary.

4) Consider the class of vector fields with non degenerate singularities
inside the disk and solve appropriate calculation problems in this class.

\end{document}